\documentclass[article,12pt]{amsart}

\usepackage{amsmath, amssymb, amsfonts, amsthm, latexsym}
\usepackage{graphicx, tikz}

\newtheorem{Theorem}{Theorem}[section] 
\newtheorem{lem}[Theorem]{Lemma}
\newtheorem{cor}[Theorem]{Corollary}
\newtheorem{prop}[Theorem]{Proposition}
\newtheorem{ex}[Theorem]{Example}

\def\core{\operatorname{core}} 
\def\Ass{\operatorname{Ass}}

\def\depth{\operatorname{depth}} 
\def\astab{\operatorname{astab}}

\def\NN{{\mathbb N}}
\def\ZZ{{\mathbb Z}} 
\def\a{{\mathbf a}}
\def\b{{\mathbf b}}
\def\c{{\mathbf c}}
\def\e{{\mathbf e}}

\def\mm{{\mathfrak m}}
\def\nn{{\mathfrak n}}
\def\F{{\mathcal F}}
\def\H{{\Omega}}

\def\D{{\Delta}}
\def\G{{\Gamma}}

\begin{document}

\title{Saturation and associated primes\\ of powers of edge ideals}
\author{Ha Thi Thu Hien}
\address{Foreign Trade University, 91 Chua Lang, Hanoi, Vietnam}
\email{thuhienha504@gmail.com}

\author{Ha Minh Lam}
\address{Institute of Mathematics, Vietnam Academy of Science and Technology, 18 Hoang Quoc Viet, Hanoi, Vietnam}
\email{hmlam@math.ac.vn}

\author{Ngo Viet Trung}
\address{Institute of Mathematics, Vietnam Academy of Science and Technology, 18 Hoang Quoc Viet, Hanoi, Vietnam}
\email{nvtrung@math.ac.vn}

\thanks{The research is supported by  the National Foundation for Science and Technology Development. 
It was carried out during the stay of the last two authors at Vietnam Institute for Advanced Study in Mathematics in 2013.}

\begin{abstract}
For the edge ideal $I$ of an arbitrary simple graph $\G$ we describe the monomials of the saturation of $I^t$ in terms of (vertex) weighted graphs associated with the monomials. This description allows us to characterize the embedded associated primes of $I^t$ as covers of $\G$ which contain certain types of subgraphs of $\G$.  
As an application, we completely classify the associated primes of $I^2$ and $I^3$ in terms of $\G$.
\end{abstract}	

\maketitle

\section*{Introduction}

Let $\G$ be a simple graph on the vertex set $V= \{1,2, \ldots,n\}$. The {\em edge ideal} of $\G$ is the ideal $I(\G)$generated by the monomials $x_ix_j$, $\{i,j\} \in \G$, in the polynomial ring $R = k[x_1,...,x_n]$ over a field $k$.  For simplicity we set $I = I(\G)$. \smallskip

The first goal of this paper is to describe the monomials of the saturation $\widetilde {I^t}$ of any power $I^t$, $t \ge 2$. The motivation comes from the facts that for every $\a \in \NN^n$, the $\a$-component of the 
local cohomology modules of $R/I^t$ can be computed by means of a simplicial complex $\D_\a$ and that the facets of $\D_\a$ can be described by the condition $x^\a \in  \widetilde {J^t} \setminus J^t$, where $x^\a$ denotes the monomial whose exponent vector is $\a$ and $J$ is the edge ideal of a subgraph of $\G$ (see Section 1 for more details). If we know the monomials of $\widetilde {J^t}$ we will be able to compute these local cohomology modules, which provides information on the depth and the Castelnuovo-Mumford regularity of $R/I^t$. In particular, the existence of  a monomial in $\widetilde {I^t} \setminus I^t$ is a criterion for the maximal homogeneous ideal to be an associated prime of $I^t$. Therefore,  using localization we can obtain a combinatorial characterization of the associated primes of $I^t$,  which is the second goal of this paper.
\smallskip

There have been several works on the behavior of the associated primes and the depth of $I^t$ for $t$ large enough (see  e.g. \cite{Br}, \cite{CMS}, \cite{HM}, \cite{HH1}, \cite{HQ}, \cite{MMV},  \cite{SVV}, \cite{TNT}), but little is known about a particular power $I^t$ (except for its Cohen-Macaulay property \cite{RTY2}, \cite{TT1}). There was a combinatorial description of the associated primes of every power of the cover ideals of graphs by Francisco, Ha and Van Tuyl \cite{FHV}. However, this result cannot be used to study edge ideals of graphs. 
Recently, Herzog and Hibi \cite{HH2} gave a criterion for $\depth R/I^2 = 0$ or, in other words, 
for the maximal homogeneous ideal to be an associated prime of $I^2$ in terms of $\G$. 
A more general result was found independently by Terai and Trung \cite{TT2} who gave a combinatorial characterization of the associated primes of the second power of an arbitrary squarefree monomial ideal. 
So it is a challenge to find a similar characterization of the associated primes of $I^t$ for $t \ge 3$. 
The results of this paper provide a general approach to this problem. \smallskip

Our idea is to represent a monomial $x^\a$ by a (vertex) weighted graph $\G_\a$, which is obtained from $\G$ by assigning every vertex $i$ with the weight $a_i$, where $a_i$ is the exponent of $x_i$ in $x^\a$. 
Our first result is a combinatorial criterion for $x^\a  \in \widetilde{I^t} \setminus I^t$ in terms of $\G_\a$. More precisely, we show that $x^\a \in \widetilde{I^t} \setminus I^t$ if and only if  the matching numbers of $\G_\a$ and of certain induced subgraphs of $\G_\a$ satisfy some bounds (see Theorem \ref{monomial}). These bounds impose strong conditions on the set $V_\a := \{i \in V|\ a_i > 0\}$. For instance, we can show that every vertex of $V \setminus V_\a$ is adjacent to a vertex of $V_\a$ and that every connected component of the induced subgraph of $\G$ on $V_\a$ contains an odd cycle of length $\le 2t-1$. In particular, we can show that $\deg x^\a \le 3(t-1)$ if $x^\a \in \widetilde{I^t} \setminus I^t$. This might have consequences on the Castelnuovo-Mumford regularity of $R/I^t$. \smallskip

It is known that an associated prime of $I^t$ is of the form $P_F$, where $F$ is a cover of $\G$ and $P_F$ denotes the ideal generated by the variables $x_i$, $i \in F$. In particular, the minimal associated primes of $I^t$ correspond to the minimal covers of $\G$. Using the description of the monomials in $\widetilde{I^t} \setminus I^t$ and the technique of localization we are able to give a criterion for $P_F$ to be an embedded (i.e. non-minimal) associated prime of $I^t$ in terms of certain weighted graphs related to $F$ (see Theorem \ref{associated}). As a consequence, we obtain a simple necessary condition for  $P_F$ to be an embedded associated prime of $I^t$, namely that $F$ is minimal among the covers of $\G$ containing the closed neighborhood of a set $U \subseteq V$ such that every connected components of  $\G_U$ contains at least one odd cycle of length $\le 2t-1$. \smallskip

In a similar manner, we can also give a sufficient condition for $P_F$ to be an embedded associated prime of $I^t$, which depends only on the existence of a special weighted graph on $\G$ (see Theorem \ref{strong}).  It turns out that the above necessary condition is also sufficient for $P_F$ to be an embedded associated prime of some power of $I$. Moreover, we can give an upper bound for the least number $t_0$ such that $P_F$ is an associated prime of $I^t$ for all $t \ge t_0$ (see Theorem \ref{construction}). The proof is inspired by the technique of adding edges to an odd cycle by Chen, Morey and Sung \cite{CMS}, who used it to give a characterization of  the stable set $\Ass^\infty(I)$ and an upper bound for the stability index $\astab(I)$ of the sets of associated primes of $I^t$. Our approach yields a simpler characterization of $\Ass^\infty(I)$ and a better upper bound for $\astab(I)$. \smallskip

Finally, to demonstrate the efficiency of our approach we show that the afore mentioned results on the associated primes of $I^2$ of Herzog and Hibi \cite{HH2} and of Terai and Trung \cite{TT2}  are immediate consequences of the above criterion.  Furthermore, we give a complete classification of the associated primes of $I^3$. According to this classification, $P_F$ is an embedded associated prime of $I^3$ if and only $F$ is a minimal cover or minimal among the covers of $\G$ containing the closed neighborhood of a subgraph of the following forms: a triangle, a union of an edge and a triangle meeting at a vertex, a union of two non-adjacent triangles, a union of two triangles meeting at a vertex, a pentagon.

\section{Saturation of monomial ideals}

Let $\mm$ be the maximal homogeneous ideal of $R =  k[x_1,...,x_n]$. 
Given an ideal $I$ of $R$,  one calls the ideal  $\widetilde I : = \bigcup_{m \ge 1}I: \mm^m$ the {\em saturation} of $I$. 
In this section we will explain why the condition $x^\a \in \widetilde I \setminus I$ is important for the computation of the local cohomology modules of a monomial ideal and for the characterization of associated primes of powers of edge ideals. \smallskip

Let $I \neq 0$ be a monomial ideal. Then $R/I$ is an $\NN^n$-graded algebra. 
Hence the local cohomology modules $H_\mm^i(R/I)$ are $\ZZ^n$-graded. 
Takayama \cite{Ta} showed that for every degree $\a \in \ZZ^n$, the $\a$-component of $H_\mm^i(R/I)$ can be expressed  in terms of the reduced homology $\tilde H_j(\D_\a,k)$ of a simplicial complex $\D_\a$, which is defined as follows. \smallskip

Let $V = \{1,...,n\}$ and $G_\a := \{i \in V|\ a_i < 0\}$. For every subset $G \subseteq V$ let $\a_G$ denote the vector obtained from $\a$ by setting $a_i = 0$ for $i \in G$ and define 
$$I_G := k[x_i|\ i \in V\setminus G] \cap IR[x_i^{-1}|\ i \in G].$$
Then the set of  the facets of $\D_\a$ is given by the formula
\begin{equation*}
\F(\D_\a)  = \{G \setminus G_\a|\ G_\a \subseteq G \subseteq V,\  x^{\a_G} \in \widetilde {I_G} \setminus I_G\}.
\end{equation*} 
This description of $\D_\a$ is taken from \cite [Lemma 1.3]{TT2}, which is simpler than the original definition in \cite{Ta}. 
\smallskip

Let $\D$ be the simplicial complex on $V$ such that  $\sqrt{I}$ is the Stanley-Reisner ideal of $\D$.
This means  $\sqrt{I}$ is generated by the monomials $x_{i_1}\cdots x_{i_r}$, $\{i_1,...,i_r\} \not \in \D$.
 For $j = 1,...,n$, we denote by $\rho_j$ the maximum of positive $j$-th coordinates of all vectors $\b \in \NN^n$ such that $x^\b$ is a minimal generator of $I$. Then the result of Takayama can be formulated as follows.

\begin{Theorem} \label{Takayama} \cite[Theorem 1]{Ta}
$$\dim_kH_\mm^i(R/I)_\a = 
\begin{cases}
\dim_k\widetilde H_{i-|G_\a|-1}(\D_\a,k) & \text{\rm if }\ G_\a \in \D\ \text{\rm and}\\
&\ \ \  \ a_j  < \rho_j\ \text{\rm for}\ j=1,...,n,\\
0 & \text{\rm else. }
\end{cases} $$
\end{Theorem}

To compute the facets of $\D_\a$ we have to check the condition $x^{\a_G} \in \widetilde {I_G} \setminus I_G$.
In particular, this condition implies $\widetilde {I_G} \neq I_G$. 
We shall see that a set $G$ with the property $\widetilde {I_G} \neq I_G$ corresponds to an associated prime of $I$.
For a subset $F$ of $V$, we denote by $P_F$ the ideal of $R$ generated by the variables $x_i$, $i \in F$.
Then every associated prime of $I$ has the form $P_F$.

\begin{lem} \label{asso}
$\widetilde {I_G} \neq I_G$ if and only if $P_{V \setminus G}$ is an associated prime of $I$.
\end{lem}

\begin{proof}
Set $A := k[x_i|\ i \in V \setminus G]$ and let $Q$ be the maximal homogeneous ideal of $A$. 
Then $\widetilde {I_G} \neq I_G$ if and only if $Q$ is an associated prime of $I_G$.
Since $P_{V \setminus G} = Q R$ and $R = A[x_i|\  i \in G]$ is a polynomial ring over $A$,  
$Q$ is an associated prime of $I_G$ if and only if $P_{V \setminus G}$ is an associated prime of $I_G R$.
Set $B := R[x^{-1}_i|\  i \in G]$. Since $B$ is a localization of $R$ and $P_{V \setminus G} \neq B$, $P_{V \setminus G}$ is an associated prime of $I_G R$ if and only if  $P_{V \setminus G}B$ is an associated prime of $I_GB$. 
By definition, $I_G$ is generated by the monomials obtained from the monomials of $I$ by removing the variables $x_i$, $i \in G$. Thus, every monomial of $I$ is divisible by a monomial of $I_G$. 
Hence $IB \subseteq I_GB$. On the other hand, $I_GB \subseteq IB$ because $I_G = A \cap IB$. Therefore, $I_GB = IB$. Since $P_{V \setminus G}B$ is an associated prime of $IB$ if and only if $P_{V \setminus G}$ is an associated prime of $I$, we get the conclusion.
\end{proof}

By Lemma \ref{asso}, we only need to check the condition $x^{\a_G} \in \widetilde {I_G} \setminus I_G$ for $G$ with the property that $P_{V\setminus G}$ is an associated prime of $I$. Therefore, the associated primes of $I$ play an important role in the computation of the local cohomology modules of $R/I$. \smallskip 

From now on let $I$ be the edge ideal of a simple graph $\G$ on the vertex set $V$. We shall see that the description of the associated primes of a power $I^t$ can be reduced to the problem when $\mm$ is an associated prime of $I^t$. It is well known that $P_F$ is a minimal associated prime of $I^t$ if and only if $F$ is a minimal (vertex) cover of  $G$. So we have to find the associated primes of $I^t$ among the ideals $P_F$, where $F$ is a cover of $\G$. \smallskip

Let $\core(F)$ denote the set of vertices in $F$ which are not adjacent to any vertex in $V \setminus F$.
Note that a cover $F$ is minimal if and only if $\core(F) = \emptyset$.
Let $\G_U$ denote the induced subgraph of $\G$ on subset $U$ of $V$. 
In the following we set $I(\G_{\core(F)}) = 0$ if  $\core(F) = \emptyset$.

\begin{prop} \label{core}
Let $F$ be a cover of $\G$. Let $J = I(\G_{\core(F)})$. Let $\nn$ denote  the maximal homogeneous ideal of $k[x_i|\ i \in \core(F)]$. Then $P_F$ is an associated prime of $I^t$ if and only if $\nn$ is an associated prime of $J^t$.
\end{prop}

\begin{proof}
Set $A := k[x_i|\ i \in F]$ and let $Q$ be the maximal homogeneous ideal of $A$.
By the proof of Lemma \ref{asso}, $P_F$ is an associated prime of $I^t$ if and only if $Q$ is an associated prime of $(I^t)_G$, where $G = V \setminus F$.
By definition, $I_G$ and  $(I^t)_G$ are generated by the monomials obtained from the monomials of $I$ and $I^t$ by removing the variables $x_i$, $i \not\in F$. From this it follows that $(I^t)_G = (I_G)^t$.
For all $j \in F \setminus \core(F)$, there exists $i \not\in F$ adjacent to $j$. Since $x_ix_j \in I$, we get $x_j \in I_G$. 
The monomials of $I_G$ which do not contain any variable $x_j$, $j \in F \setminus \core(F)$, belong to $J$.
Since $J \subset I_G$, this implies 
$$I_G = \big(x_j|\ j \in F \setminus \core(F)\big)A+ JA.$$
Note that the ideals  $\big(x_j|\ j \notin F \setminus \core(F)\big)$ and $J$ are generated 
by monomials in two disjoint sets of variables. Then $Q$ is an associated prime of $(I_G)^t$ if and only if $\nn$ is an associated prime of $J^s$ for some $s \le t$ \cite[Lemma 2.1]{CMS}.  By \cite[Theorem 2.15]{MMV},  the latter condition implies that $\nn$ is also an associated prime of $J^t$. Therefore, $P_F$ is an associated prime of $I^t$ if and only if $\nn$ is an associated prime of $J^t$. 
\end{proof}
 
By Lemma \ref{core}, the description of the associated primes of $I^t$ can be reduced to problem when there does exist a monomial $x^\a \in \widetilde{J^t} \setminus J^t$ because this condition is a criterion for $\nn$ to be 
an associated prime of $J^t$. Furthermore, the condition $x^{\a_G} \in \widetilde {(I^t)_G} \setminus (I^t)_G$, which appears in the definition of the complex $\D_\a$ of $I^t$, is in fact equivalent to the condition $x^\a \in  \widetilde{J^s} \setminus J^s$ for some edge ideal $J$ and some $s \le t$.

\begin{prop}  
Let $F = V \setminus G$ and $J = I(\G_{\core(F)})$. Then   $x^{\a_G} \in \widetilde {(I^t)_G} \setminus (I^t)_G$ if and only if $x^{\a_{V \setminus \core(F)}} \in \widetilde{J^s} \setminus J^s$, where $s = t-\sum_{i \in F \setminus \core(F)}a_i$.
\end{prop}

\begin{proof}
By the proof of Proposition \ref{core} we have $(I^t)_G = (I_G)^t$ and
$$I_G = \big(x_j|\ j \in F \setminus \core(F)\big)A+ JA.$$
This formula implies the following relations:\par
(1) $(I_G)^t: x_j = (I_G)^{t-1}$ for all $j \in F \setminus \core(F)$,\par
(2) $(I_G)^t \cap S = J^t$, where $S = k[x_i|\ i \in \core(F)]$. \par
\noindent Note that $x^{\a_G} = x^{\a_{V \setminus \core(F)}}\prod_{j \in F \setminus \core(F)}x_j^{a_j}.$
Then $x^{\a_G} \in  (I_G)^t$ if and only if 
$x^{\a_{V \setminus \core(F)}} \in (I_G)^t: \prod_{j \in F \setminus \core(F)}x_j^{a_j} = (I_G)^s$,
where the last equality follows from (1). By definition, $x^{\a_{V \setminus \core(F)}} \in S$.
It follows from (2) that $x^{\a_{V \setminus \core(F)}} \in (I_G)^s$ if and only if 
 $x^{\a_{V \setminus \core(F)}}\in J^s$. Therefore, $x^{\a_G} \in  (I_G)^t$ if and only if  $x^{\a_{V \setminus \core(F)}} \in J^s$. 
\par
Similarly, we can show that for $i \in \core(F)$, $x^{\a_G} \in  \bigcup_{m \ge 1}(I_G)^t: x_i^m$ if and only if  $x^{\a_{V \setminus \core(F)}} \in \bigcup_{m\ge 1}J^s:x_i^m$. Since $\bigcup_{m \ge 1}(I_G)^t: x_i^m =  A$ for $i \in F \setminus \core(F)$, we have
$$\widetilde {(I_G)^t} = \bigcup_{m \ge 1}(I_G)^t: Q^m =
\bigcup_{m \ge 1}\bigcap_{i \in F}(I_G)^t: x_i^m = \bigcup_{m \ge 1}\bigcap_{i \in \core(F)}(I_G)^t: x_i^m.$$
Therefore, $x^{\a_G} \in \widetilde {(I_G)^t}$ if and only if 
$$x^{\a_{V \setminus \core(F)}} \in \bigcup_{m\ge 1}\bigcap_{i \in \core(F)} J^s:x_i^m = 
\bigcup_{m\ge 1}J^s:\nn^m =\widetilde{J^s}.$$
So we can conclude that $x^{\a_G} \in \widetilde {(I_G)^t} \setminus (I_G)^t$ 
if and only if $x^{\a_{V \setminus \core(F)}} \in \widetilde {J^s} \setminus J^s$.
\end{proof}

\section{Saturating weighted graphs}

Let $I$ be the edge ideal of a simple graph $\G$ on the vertex set $V = \{1,...,n\}$. The aim of this section is to characterize the condition $x^\a \in  \tilde{I^t} \setminus I^t$ for a given vector $\a \in \NN^n$ and to describe the associated primes of $I^t$ in  terms of $\G$. \smallskip

We shall need the notion of (vertex) weighted graphs. A {\em weighted graph} is a simple graph whose vertices are assigned with positive integers called {\em weight}. The simple graph is called the {\em base graph} of the weighted graph. We always consider a simple graph as a weighted graph whose vertices have the trivial weight 1.
\smallskip

Let $\H$ be an arbitrary weighted graph on a vertex set $U$. 
A {\em matching} of $\H$ is a family of edges (not necessarily different) such that 
every vertex of $\H$ appears in these edges no more times than its weight. 
The maximal number of edges of the matchings of $\H$ is called the {\em matching number}, 
denoted by $\nu(\H)$.  For every subset $N$ of $U$ we denote by $\H - N$ the  induced weighted subgraph of $\H$ on the vertex set $U \setminus N$.
 \smallskip
 
Let $V_\a := \{i \in V|\ a_i > 0\}$. If we assign each vertex $i \in V_\a$ with the weight $a_i$, we obtain a weighted graph called $\G_\a$. 
For a vertex $i \in V$, we denote by $N_\a(i)$ the set of all adjacent vertices of $i$ in $\G_\a$ and set $\deg_{\a}(i) = \sum_{j \in N_\a(i)} a_j$.  Using these notions we can translate the condition $x^\a \in  \tilde{I^t} \setminus I^t$
in combinatorial terms as follows. 

\begin{Theorem} \label{monomial} 
$x^\a \in  \widetilde{I^t} \setminus I^t$ if and only if the following conditions are satisfied: \par
{\rm (i) } $\nu(\G_\a) < t$, \par
{\rm (ii)} $\nu(\G_\a  - N_\a(i)) \ge t-\deg_\a(i)$ for all $i \in V$. 
\end{Theorem}

\begin{proof}
First we will show that $x^\a \not\in I^t$ if and only if condition (i) is satisfied.
Let $\e_i$ be the $i$-th unit vector of $\NN^n$. It is clear that $x^\a \in I^t$ if and only if there exists a family of $t$ edges $\{i_1,j_1\},...,\{i_t,j_t\}$ of $\G$ (not necessarily different) such that $x^\a$ is divisible by the product
$(x_{i_1}x_{j_1}) \cdots (x_{i_t}x_{j_t})$. The divisibility implies that every vertex of $V_\a$ appears in these edges no more times than its weight. Hence these edges form a matching of $\G_a$.
Thus, $x^\a \in I^t$ if and only if $\nu(\G_\a) \ge t$.  \par

It remains to show that $x^\a \in \widetilde{I^t}$ if and only if condition (ii) is satisfied.
By definition, $x^\a \in \widetilde{I^t}$ if and only if $x_i^m x^\a \in I^t$ for all $i \in V$ and $m \gg 0$.  
As we have seen above, this is equivalent to the condition $\nu(\G_{\a+ m\e_i}) \ge t$.
Let $\H$ be the weighted subgraph of $\G_{\a+m\e_i}$ 
whose edges contain at least one vertex in $N_\a(i)$ 
and whose vertices have the same weight as in $\G_{\a+m\e_i}$. 
Then the number of edges of a matching of $\H$ can not exceed the appearing times of the vertices of $N_\a(i)$ in these edges.  Hence $\nu(\H) \le \sum_{j \in N_\a(i)}a_j = \deg_\a(i)$. 
If $m \ge \deg_\a(i)$, there is a matching of $\H$ which consists of $\deg_\a(i)$ edges connecting $i$ with every vertex $j \in N_\a(i)$ $a_j$ times. Thus, $\nu(\H) = \deg_\a(i)$. 
Since the union of this matching of $\H$ with any matching of $\G_\a  - N_\a(i)$ is a matching of $\G_{\a+m\e_i}$, 
$$\nu(\G_{\a+m\e_i}) \ge \deg_\a(i) + \nu(\G_\a  - N_\a(i))$$
for $m \gg 0$. On the other hand, since every matching of $\G_{\a+m\e_i}$ is the disjoint union of a matching of  $\H$ and a matching of $\G_\a  - N_\a(i)$, 
$$\nu(\G_{\a+m\e_i}) \le \nu(\H) + \nu(\G_\a  - N_\a(i)) = \deg_\a(i) + \nu(\G_\a  - N_\a(i)).$$ 
Therefore, $\nu(\G_{\a+m\e_i}) = \deg_\a(i) + \nu(\G_\a  - N_\a(i))$ for $m \gg 0$. So we can conclude that 
$x^\a \in \widetilde{I^t}$ if and only if $\deg_{\G_\a}(i) + \nu(\G_\a  - N_\a(i)) \ge t$.
\end{proof}

In graph theory a subset $W$ of $V$ (or a subgraph of $\G$ on $W$) is called {\em dominating} if every vertex of $V \setminus W$ is adjacent to at least one vertex of $W$.

\begin{cor} \label{dominating}
If $x^\a \in  \widetilde{I^t} \setminus I^t$, then $V_\a$ is a dominating set of $\G$.
\end{cor}

\begin{proof}
By Theorem \ref{monomial} we have 
$$\deg_{\a}(i) \ge t - \nu(\G_\a  - N_\a(i)) \ge t - \nu(\G_\a) \ge t- (t-1) =1$$
for all $i \in V \setminus V_\a$. This implies that $i$ is adjacent to at least one vertex in $V_\a$.
\end{proof}

\begin{ex}
{\rm Let $\G$ be a graph which contains a dominating cycle $C$ of length $2t-1$, $t \ge 2$. 
Let $\a \in \{0,1\}^n$ such that $V_\a$ is the vertex set of $C$. Then $x^\a \in  \tilde{I^t} \setminus I^t$.
To see this we will show that the conditions (i) and (ii) of Theorem \ref{monomial} are satisfied.
First, we observe that $\G_\a$ is the induced subgraph of $\G$ on $V_\a$.
Since $|V_\a| = 2t-1$, we must have $\nu(\G_\a) < t$.
Since $V_\a$ is a dominating subset of $V$, every vertex $i \in V\setminus V_\a$ is adjacent to some vertex of $V_\a$. This property also holds for $i \in V_\a$ because every vertex of $V_\a$ is adjacent to two vertices of $C$.
Therefore, $N_\a(i)$ is a non-empty subset of $V_\a$ for all $i \in V$. Since $|N_\a(i)| = \deg_\a(i)$, 
we have $\deg_\a(i) > 0$. Since every vertex of $N_\a(i)$ is contained in two edges of $C$,  the subgraph $\G_\a - N_\a(i)$ contain at least $2t-1-2\deg_\a(i)$ edges of $C$. Since these edges does not form the cycle $C$,
there are at least $t-\deg_\a(i)$ disjoint edges among them. 
Therefore, $\nu(\G_\a  - N_\a(i)) \ge t-\deg_\a(i)$ for all $i \in V$.}  
\end{ex}

We note that condition (ii) of Theorem \ref{core} may involve vertices outside of $\G_\a$. To investigate the conditions involving only the vertices of $\G_\a$ we call a weighted graph $\H$ on a vertex set $U$ {\em $t$-saturating} if $\nu(\H) < t$ and $\nu(\H - N_\H(i)) \ge t - \deg_\H(i)$ for all $i \in U$,  
where $N_\H(i)$ denotes the set of all adjacent vertices of $i$ and $\deg_\H(i)$ is the sum of the weights of the vertices of $N_\H(i)$.  
\smallskip 

Since $N_{\G_\a}(i) = N_\a(i)$ and $\deg_{\G_\a}(i) = \deg_\a(i)$ for $i \in V_\a$, Theorem \ref{monomial} can be reformulated as {\em  $x^\a \in  \tilde{I^t} \setminus I^t$ if and only if $\G_\a$ is $t$-saturating and $\nu(\G_\a  - N_\a(i)) \ge t - \deg_\a(i)$ for all $i \in V \setminus V_\a$}.   
\smallskip

It is well known that $\mm$ is an associated prime of $I^t$ if and only if there exists a monomial $x^\a \in  \widetilde{I^t} \setminus I^t$. Therefore, one can reformulate Theorem \ref{monomial} as a criterion for $\mm$ to be an associated prime of $I^t$. In the following we  will deduce from this criterion a combinatorial characterization of the embedded associated primes of $I^t$. Note that an embedded prime of $I^t$ is of the form $P_F$, where $F$ is a cover of $\G$ with  $\core(F) \neq \emptyset$. \smallskip

For a subset $U$ of $V$ we denote by $N(U)$ the set of the vertices adjacent to some vertex of $U$ and by $N[U]$ the union of $U$ and $N(U)$. These sets are called the {\em open neighborhood} and the {\em closed neighborhood} of $U$ in $\G$.
 
\begin{Theorem} \label{associated} 
Let $F$ be a cover of $\G$ with $\core(F) \neq \emptyset$. Then $P_F$ is an embedded associated prime ideal of $I^t$ if and only if $F$ is minimal among the covers containing $N[V_\a]$ for some $\a \in \NN^n$ such that  $\G_\a$ is $t$-saturating and $\nu(\G_\a  - N_\a(i)) \ge t - \deg_\a(i)$ for all $i \in \core(F) \setminus V_\a$.  
\end{Theorem}

\begin{proof}
By Proposition \ref{core}, $P_F$ is an associated prime ideal of $I^t$
 if and only if $\widetilde{J^t} \neq J^t$, where $J = I(\G_{\core(F)})$. 
The latter condition means that there exists $\a \in \NN^n$ 
with $V_\a \subseteq \core(F)$ such that $x^\a \in \widetilde{J^t} \setminus J^t$. 
Obviously, $(\G_{\core(F)})_{V_\a} = \G_{V_\a}$ and $(\G_{\core(F)})_\a = \G_\a$. 
By Theorem \ref{monomial}, $x^\a \in \widetilde{J^t} \setminus J^t$  
 if and only if $\G_\a$ is $t$-saturating and $\nu(\G_\a  - N_\a(i)) \ge t - \deg_\a(i)$ for all $i \in \core(F) \setminus V_\a$. It remains to show that the condition $V_\a \subseteq \core(F)$ can be replaced by the condition $F$ is minimal among the covers containing $N[V_\a]$. \par
If $V_\a \subseteq \core(F)$, then $N[V_\a] \subseteq F$ by the definition of $\core(F)$. As we have seen above, we may assume further assume that $x^\a \in \widetilde{J^t} \setminus J^t$. Then  $V_\a$ is a dominating set of $\G_{\core(F)}$ by Corollary \ref{dominating}. Hence $\core(F) \subseteq N[V_\a]$.  Thus, every vertex of $F \setminus N[V_\a]$ is adjacent to a vertex of $V \setminus F$. From this it follows that every set $F \setminus i$, $i \in F \setminus N[V_\a]$, is not a cover of $\G$. Therefore, $F$ is minimal among the covers containing $N[V_\a]$. Conversely,  if  $F$ is minimal among the covers containing $N[V_\a]$, then $N[V_\a] \subseteq F$. Hence every vertex of $V_\a$ is not adjacent to any vertex of $V \setminus F$. This implies $V_\a \subseteq \core(F)$. 
\end{proof}

By Theorem \ref{associated},  to find the embedded associated primes of $I^t$ we have  first to find the subsets $U \subseteq V$ such that there exists a $t$-saturating weighted graph on $U$. 
In the following we establish properties of $t$-saturating weighted graphs which can be used to detect them.  
\smallskip

In the following we denote by $\H$ a weighted graph on the vertex set $U$ and by $a_i$ the weight of a vertex $i \in U$. Recall that a vertex of is called a {\em leaf vertex} if it is adjacent to only a vertex of the graph.

\begin{lem} \label{weight} 
Let $\H$ be a $t$-saturating weighted graph. Then \par
{\rm (i) } $a_i < \min\{\deg_\H(i), \nu(\H)+1\}$, \par
{\rm (ii)} $a_i \ge 2$ if $i$ adjacent to a leaf vertex of the base graph of $\H$. 
\end{lem}

\begin{proof}
One can find a family of  $\min\{a_i, \deg_\H(i)\}$ edges of $\H$ containing $i$ such that  the appearing times of every vertex $j \in N_\H(i)$ does not exceed $a_j$. These edges can be added to an arbitrary matching of $\H - N_\H(i)$ to form a matching of $\H$. Therefore,
$$\min \{a_i, \deg_\H(i)\} + \nu(\H - N_\H(i)) \le \nu(\H).$$
By the definition of $t$-saturating weighted graph, $\deg_\H(i)+\nu(\H - N_\H(i) \ge t > \nu(\H)$.
Therefore, we must have  $a_i < \deg_\H(i)$ and $a_i + \nu(\H - N_\H(i)) \le \nu(\H)$. The latter inequality implies $a_i  <  \nu(\H)+1$. So we obtain (i).
If $i$ is adjacent to a leaf vertex $j$, then $a_i = \deg_\H(j)$. 
As shown above, $\deg_\H(j) > a_j$. Since $a_j \ge 1$, this implies (ii). 
\end{proof}

For the proof of the next result we need to extend some notations on simple graphs to weighted graphs.
Let $M$ be a matching of a weighted graph $\H$. 
For every vertex $i$ let $w_i$ be the appearing times of $i$ in the edges of $M$. 
We say that $i$ is an {\em unmatched vertex} of $M$ if $w_i < a_i$.
An {\em augmenting walk} of $M$ is a sequence of vertices and edges, where each edge's endpoints are the preceding and following vertices in the sequence, which satisfy the following conditions:\par
(i) The first and last vertices are unmatched vertices,\par
(ii) The appearing time of every vertex $i$ does not exceed $a_i$,\par
(iii) The number of edges is odd and the family of the even edges is contained in $M$.\par
With these notations we can easily extend Berge's theorem (see e.g. \cite[Theorem 1.2.1]{LM}) to weighted graph, which says that $|M| = \nu(\H)$ if and only if $M$ does not have any augmenting walk. We leave the reader to check this fact.

\begin{prop} \label{bound} 
Let $\H$ be a $t$-saturating weighted graph. Then $\sum_{i \in U}a_i \le 3(t-1)$.
\end{prop}

\begin{proof}
Let $\G$ be the base graph of $\H$ and $\a$ the weight vector of $\H$. 
Then $\G_\a = \H$ and $N_\a(i) = N_\H(i)$ and $\deg_\a(i) = \deg_\H(i)$ for every $i \in U$. 
Choose a number $m \ge \deg_\a(i)$. We have seen in the proof of Theorem \ref{monomial} that
$\nu(\G_{\a + m\e_i}) = \deg_\a(i) + \nu(\G_\a - N_\a(i)).$
Since $\G_\a$ is $t$-saturating,  
$\deg_\a(i) + \nu(\G_\a - N_\a(i)) \ge t > \nu(\G_\a).$
Therefore, $\nu(\G_{\a+m\e_i}) > \nu(\G_\a)$. 
\par
Let $M$ be a matching of $\G_\a$ with $|M| = \nu(\G_\a)$. 
Consider $M$ as a matching of  $\G_{\a+m\e_i}$.
By the above mentioned weighted version of Berge's theorem, there is an augmenting walk $P$ of $M$.
Adding the odd edges of $P$ to $M$ and deleting the even edges of $P$ from $M$ we obtain a matching $M^*$ of $\G_{\a + m\e_i}$. It is clear that $M^*$ has $\nu(\G_\a)+1$ edges. 
Hence $M^*$ is not a matching of $\G_\a$. This has the following consequences on the vertex $i$.
First, since $\G_\a$ differs from $\G_{\a+m\e_i}$ only by the weight of $i$, $i$ must appear in $P$
and $w_i^* > a_i$, where $w_j^*$ denote the appearing times of a vertex $j$ in $M^*$.  
Let $i_f$ and $i_l$ denote the first and the last vertex of $P$.
Then $w_j^* = w_j \le a_j$ if $j \neq i_f, i_l$.
Therefore, $i = i_f$ or $i = i_l$. Without restriction we may assume that $i = i_f$.
If $i \neq i_l$, then $w_i^* = w_i+1 \le a_i+1$. Together with the condition $w_i^* > a_i$, this implies $w_i = a_i$.
Thus, if $w_i < a_i$, we must have $i = i_f = i_l$. In this case, $w_i^* = w_i+2 \le a_i +1$. Hence $w_i^* = a_i+1$, $w_i = a_i-1$, and $w_j^* = w_j \le a_j$ for $j \neq i$.
\par

Let $W$ denote the set of the vertices $i \in U$ with $w_i = a_i-1$. Then $w_i = a_i$ for $i \not\in W$.
Therefore,
$$\sum_{i \in U} a_i = \sum_{i \in U} w_i + |W| = 2|M| + |W| = 2\nu(\G_\a)+|W|.$$
The vertices of $W$ are not adjacent to each other because otherwise we could add any edge of $\G$ with endpoints in $W$ to $M$ to obtain a matching of $\G_\a$ with $\nu(\G_\a) +1$ edges.
For $i \in W$ we denote by $E_i$ the first edge of the augmenting walk $P$.
As we have seen above, $i \in E_i$. Let $h$ be the other endpoint of $E_i$.
Then every other vertex $j \in W$ is not adjacent to $h$ because otherwise we could replace $E_i$ by the edge
$\{j,h\}$ to obtain from $M^*$ a matching of $\G_\a$ with $\nu(\G_\a) +1$ edges.
This follows from the facts that $w_i^*  = a_i+1$, $w_j^* = w_j = a_j-1$. Therefore, $E_i \cap E_j = \emptyset$. From this it follows that the edges $E_i$, $i \in W$, form a matching of $\G$. 
Hence $|W| \le \nu(\G) \le \nu(\G_\a)$.  
So we obtain $\sum_{i \in U} a_i  \le 3\nu(\G_\a) \le 3(t-1)$. 
\end{proof}

\noindent {\em Remark}. The bound $\sum_{i \in U}a_i \le 3(t-1)$ is sharp as can be seen from the case $\H$ is the simple graph of $t-1$ disconnected triangles. \smallskip

By Theorem \ref{monomial}, we immediately obtain the following interesting consequence.

\begin{cor} 
If $x^\a \in  \widetilde{I^t} \setminus I^t$, then $\deg x^\a \le 3(t-1)$.
\end{cor}

\begin{lem} \label{odd cycle} 
The base graph of any $t$-saturating weighted graph contains at least one odd cycle of length $\le 2t-1$.  
\end{lem}

\begin{proof}
Let $\G$ be the base graph of a $t$-saturating weighted graph $\H$.
Then there is $\a \in \NN^n$ such that $\H = \G_\a$.
Since $V = V_\a$, $x^\a \in  \widetilde{I^t} \setminus I^t$ by Theorem \ref{monomial}. 
Hence $\widetilde{I^t} \neq I^t$. Let $I^{(t)}$ denotes the $t$-th symbolic power of $I$, which is the intersection of the primary components of the minimal associated primes of $I^t$. Then $I^{(t)} \supseteq \widetilde{I^t}$. Hence $I^{(t)} \neq I^t$. By \cite[Lemma 3.10]{RTY1} (see also [12, Lemma 5.8, Theorem 5.9]), this implies that $\G$ has at least one odd cycle of length  $\le 2t-1$.
\end{proof}

\noindent {\em Remark}. A 2-saturating weighted graph must be the graph of a triangle because it contains a triangle by Lemma \ref{odd cycle} and any other weighted graph containing a triangle has matching number at least 2. \smallskip

By definition, if $\H$ is a $t$-saturating weighted graph, then $\H$ is $s$-saturating for $s = \nu(\H) +1$.
This property is preserved by the connected components of $\H$ in the following sense.

\begin{lem} \label{component} 
Let $\H_1,...,\H_m$ be the connected components of  a weighted graph $\H$. Let $t = \nu(\H)+1$ and $t_i = \nu(\H_i)+1$, $i = 1,...,m$. Then $\H$ is $t$-saturating if and only if $\H_i$ is $t_i$-saturating for $i = 1,...,m$.
\end{lem}

\begin{proof}
It is easy to see that $\nu(\H) = \sum_{i =1}^m \nu(\H_i)$. 
Hence $t - 1 = \sum_{i =1}^m(t_i - 1)$.
Assume that $\H$ is $t$-saturating. Then $\deg_\H(i) + \nu(\H - N_\H(i))  \ge t$ for every vertex $i$ of $\H$.
If $i$ is a  vertex of $\H_1$, we have $N_\H(i) = N_{\H_1}(i)$. Hence $\deg_{\H_1}(i) = \deg_\H(i)$.
Moreover, the connected components of $\H - N_\H(i)$ are $\H_1  - N_{\H_1}(i)$ 
and $\H_2,...,\H_m$. Therefore,  
$$\nu(\H - N_\H(i)) = \nu(\H_1  - N_{\H_1}(i)) + \sum_{j=2}^m \nu(\H_i) = \nu(\H_1  - N_{\H_1}(i)) + \sum_{j=2}^m(t_j-1).$$
From this it follows that
$$\deg_{\H_1}(i) +  \nu(\H_1  - N_{\H_1}(i)) 
=  \deg_\H(i) + \nu(\H - N_\H(i)) - \sum_{j=2}^m(t_j-1) \ge t - \sum_{j=2}^m(t_j-1) = t_1.$$
Hence $\H_1$ is $t_1$-saturating. 
Similarly, $\H_i$ is $t_i$-saturating for $i = 2,...,m$ as well.
\par

Conversely, assume that $\H_i$ is $t_i$-saturating for $i = 1,...,m$.
Let $i$ be an arbitrary vertex of $\H$.  Without loss of generality we may assume that $i$ is a vertex of $\H_1$.
Then $\deg_{\H_1}(i) +  \nu(\H_1  - N_{\H_1}(i)) \ge t_1$. Therefore,
$$\deg_\H(i) + \nu(\H - N_\H(i)) = \deg_{\H_1}(i) +  \nu(\H_1  - N_{\H_1}(j)) + \sum_{j=2}^m(t_j-1)
\ge t_1 + \sum_{j=2}^m (t_j-1) = t.$$
Hence $\H$ is $t$-saturating, as required.
\end{proof}

Using the above properties of $t$-saturating weighted graphs we obtain the following necessary condition for $P_F$ to be an embedded associated prime of $I^t$.  

\begin{Theorem} \label{embedded}
Let $P_F$ be an arbitrary embedded associated prime of $I^t$. Then $F$ is minimal among the covers of $\G$ containing $N[U]$ for a subset $U$ of $V$ such that every connected component of the induced graph $\G_U$ contains at least one odd cycle of length $\le 2t-1$. 
\end{Theorem} 

\begin{proof}
By Theorem \ref{associated}, an embedded associated prime of $I^t$ is of the form $P_F$, where $F$ is minimal among the covers of $\G$ containing $N[V_\a]$ for a $t$-saturating weighted graph $\G_\a$. By Lemma \ref{odd cycle} and Lemma \ref{component},  every connected component of the induced graph $\G_{V_\a}$ must contain at least one odd cycle of length $\le 2t-1$.
\end{proof}

It is easy to find examples showing that the above condition is not sufficient for $P_F$ to be an embedded associated prime of $I^t$.  
\smallskip

\begin{ex}
{\rm Let $\G$ be the graph of the edges 
$\{1,2\},\{1,3\},\{2,3\},\{3,4\},\{4,5\}$
and $U=\{1,2,3,4\}.$ Then $N[U] = V = \{1,2,3,4,5\}$. Hence $V$ is the only cover containing $N[U]$. However, $P_V = \mm$ is not an associated prime of $I^2$. For, if $\mm$ is an associated prime of $I^2$, then by Theorem \ref{associated}, $V$ is minimal among the covers containing $N[V_\a]$ for some $\a \in \NN^5$ such that $\G_\a$ is 2-saturating. This condition implies that $\G_\a$ is a triangle. Hence $V_\a = \{1,2,3\}$. However, $V$ is not a minimal cover containing $N[V_\a] = \{1,2,3,4\}$.}
\end{ex}

\section{Strongly saturating weighted graphs}

Let $I$ be the edge ideal of a simple graph $\G$ as in the previous section. The aim of this section is to introduce a special class of weighted graphs whose existence on $\G$ gives rise to embedded associated primes of $I^t$, $t \ge 2$. Note that the existence of a $t$-saturating weighted graph on $\G$ alone does not necessarily leads to an embedded associated prime of $I^t$. By Theorem \ref{associated} we need a further condition on vertices outside such a weighted graph.
\smallskip

We call a weighted graph $\H$ {\em strongly $t$-saturating} if $\nu(\H) < t$ and 
$\nu(\H - j) \ge t - a_j$  for all vertices $j$ of $\H$, where $a_j$ is the weight of $j$. 
There many strongly $t$-saturating graphs (whose vertices have weight 1). For instance, it is easy to check that the union of $t-1$ triangles which are disconnected or which meet each other at a common vertex is strongly $t$-saturating. Another examples are graphs on $2t-1$ vertices which contain an odd cycle of length $2t-1$.\smallskip 

Recall that $\H$ is $t$-saturating if $\nu(\H) < t$ and $\nu(\H - N_\H(i)) \ge t - \sum_{j \in N_\H(i)}a_j$ for all vertices $i$ of $\H$. Then a strongly $t$-saturating weighted graph is $t$-saturating by the following property.

\begin{lem} \label{strong property} 
Let $\H$ be a strongly $t$-saturating weighted graph $\H$. 
Let $N$ be a non-empty set of vertices of $\H$. 
Then $\nu(\H - N) \ge t - \sum_{j \in N}a_j$.
\end{lem}

\begin{proof}
Choose a vertex $i \in N$. Let $\H'$ be the weighted subgraph of $\H - i$ 
whose edges contain at least one vertex in $N \setminus i$ 
and whose vertices have the same weight as in $\H$. 
Then the number of edges of a matching of $\H'$ can not exceed the appearing times of the vertices of $N \setminus i$ in these edges.  Hence $\nu(\H') \le \sum_{j \in N}a_j - a_i$. 
Since every matching of $\H - i$ is the disjoint union of a matching of $\H'$ and a matching of $\H - N$, we have 
$$\nu(\H - i) \le \nu(\H') + \nu(\H - N) \le \sum_{j \in N}a_j - a_i  + \nu(\H - N).$$
This implies $\nu(\H - N) \ge  \nu(\H - i) + a_i - \sum_{j \in N}a_j$.
Since $\H$ is strongly $t$-saturating, $\nu(\H - i) \ge t - a_i$. 
Therefore, $\nu(\H - N) \ge t - \sum_{j \in N}a_j$.
\end{proof}

By definition, if $\H$ is a strongly $t$-saturating weighted graph, then $\H$ is strongly $s$-saturating for $s =\nu(\H)+1$. This property is preserved by the connected components of $\H$ as follows.

\begin{lem} \label{strong component} 
Let $\H_1,...,\H_m$ be the weighted connected components of  a weighted graph $\H$. 
Let $t = \nu(\H)+1$ and $t_i = \nu(\H_i)+1$, $i = 1,...,m$. 
Then $\H$ is strongly $t$-saturating if and only if 
$\H_i$ is strongly $t_i$-saturating for  $i = 1,...,m$.
\end{lem}

\begin{proof}
The proof is similar to that of Lemma \ref{component}. Hence we omit it.
\end{proof}

The following sufficient condition for an embedded associated prime of $I^t$ shows that if there exists a strongly $t$-saturating weighted graph on a subset of $V$, then we can construct an embedded prime of $I^t$.

\begin{Theorem} \label{strong} 
Let $F$ be a cover of $\G$ which is minimal among the covers containing $N[U]$ for a set $U \subseteq V$ such that there exists a strongly $t$-saturating weighted graph on $U$. Then $P_F$ is an embedded associated prime of $I^t$. 
\end{Theorem}

\begin{proof}
Let $\a \in \NN^n$ such that $V_\a = U$ and $\G_\a$ is a strongly $t$-saturating weighted graph.
By Theorem \ref{associated} it suffices to show that 
$\nu(\G_\a  - N_\a(i)) \ge t - \deg_\a(i)$ for $i \in \core(F) \setminus V_\a$. 
As in the proof of Theorem \ref{associated}, 
the assumption on $F$ implies $\core(F) \subseteq N[V_\a]$. 
Hence $i$ is adjacent to at least one vertex $j \in V_\a$. 
This means $N_\a(i) \neq \emptyset$. 
Since $\deg_\a(i) = \sum_{j \in N_\a(i)}a_j$, the above inequality follows from Lemma \ref{strong property}.
\end{proof}

We don't know whether the above sufficient condition is also necessary. \smallskip

In the following we are interested in strongly $s$-saturating subgraphs of $\G$ on $2s-1$ vertices for some $s < t$. 
The reason is that we can add edges to such a subgraph to obtain strongly $t$-saturating weighted graphs. 
The idea originates from a construction of Chen, Morey and Sung in \cite[Theorem 3.7]{CMS}.

\begin{lem} \label{existence}
Let $U$ be a subset of $V$ such that $\G_U$ is connected and contains a strongly $s$-saturating graph on $2s-1$ vertices. Then there exists  a strongly $t$-saturating weighted graph $\H$ on $U$ with $\nu(\H) = t-1$ for all $t \ge |U| - s +1$.
\end{lem}

\begin{proof} 
We can reformulate the assertion as there exists $\a \in \NN^n$ with $V_\a = U$ such that $\G_\a$ is a strongly $t$-saturating weighted graph with $\nu(\G_\a) = t-1$ for all $t \ge |U| - s +1$. 
By the assumption, there exists a strongly $s$-saturating graph $C$ on a set $W \subseteq U$ of $2s-1$ vertices. 
Since $s > \nu(C) \ge \nu(C-i) \ge s-1$, we have $\nu(C) = \nu(C - i) = s-1$ for all $i \in W$. 
Since $C$ is a subgraph of the induced graph $\G_W$, $\nu(\G_W) \ge \nu(C) = s-1$.
On the other hand, $\nu(\G_W) \le s-1$ because $\G_W$ has $2s-1$ vertices. Therefore, $\nu(\G_W) = s-1$.
Since $\nu(\G_W) \ge \nu(\G_W-i) \ge \nu(C-i) = s-1$, we also have $\nu(\G_W - i) = s-1$ for all $i \in W$.
Therefore, $\G_W$ is a strongly $s$-saturating graph. 
Let $\c \in \NN^n$ such that $\G_\c = \G_W$. 
Starting from $\G_\c$  we can build up a weighted graph $\G_\a$ as in the assertion by using the following claim.
 \smallskip

\noindent {\em Claim}.
Let $\b \in \NN^n$ such that $\G_\b$ is a strongly $t$-saturating weighted graph with $\sum_{i=1}^nb_i =2t-1$ and $\nu(\G_\b) = \nu(\G_{\b-\e_i}) = t-1$ for all $i \in V_\b$.
Let $\{h,j\}$ be an edge of $\G$ such that $h \in V_\b$. Put $\a = \b  + \e_h + \e_j$. 
Then $\G_\a$ is a strongly $(t+1)$-saturating weighted graph with $\sum_{i=1}^na_i =2t+1$ and $\nu(\G_\a) = \nu(\G_{\a-\e_i})= t$ for all $i \in V_\a$. \smallskip

For convenience we say that $\G_\a$ is obtained from $\G_\b$ by adding an edge.  

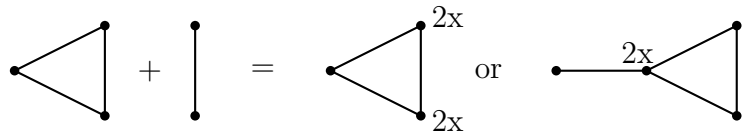
\begin{figure}[ht!]
\begin{tikzpicture}[scale=0.6]
    
\draw [thick] (-8,1) coordinate (a) -- (-6,0) coordinate (b) ;
\draw [thick] (-8,1) coordinate (a) -- (-6,2) coordinate (c) ;
\draw [thick] (-6,0) coordinate (b) -- (-6,2) coordinate (c) ; 
\draw (-5,1) node{+};
\draw [thick] (-4,0) coordinate (d) -- (-4,2) coordinate (e);
\draw (-2.5,1) node{=};
\draw [thick] (-1,1) coordinate (f) -- (1,0) coordinate (g) ;
\draw [thick] (-1,1) coordinate (f) -- (1,2) coordinate (h) ;
\draw [thick] (1,0) coordinate (g) -- (1,2) coordinate (h) ; 
\draw (1,-0.1) node[below,right] {2x};
\draw (1,2.2) node[above, right] {2x};
\draw (2.5,1) node{or};
\draw [thick] (4,1) coordinate (j) -- (6,1) coordinate (i) ;
\draw [thick] (6,1) coordinate (i) -- (8,0) coordinate (k) ;
\draw [thick] (6,1) coordinate (i) -- (8,2) coordinate (l) ;
\draw [thick] (8,0) coordinate (k) -- (8,2) coordinate (l) ;
\draw (5.8,1.4) node{2x};
     
\fill (a) circle (3pt);
\fill (b) circle (3pt);
\fill (c) circle (3pt);
\fill (d) circle (3pt);
\fill (e) circle (3pt);
\fill (f) circle (3pt);
\fill (g) circle (3pt);
\fill (h) circle (3pt);
\fill (i) circle (3pt);
\fill (j) circle (3pt);
\fill (k) circle (3pt);
\fill (l) circle (3pt);

\end{tikzpicture}
\end{figure}

\noindent {\it Proof of the claim}. 
It is clear that $\sum_{i=1}^na_i = \sum_{i=1}^nb_i + 2 = 2t +1$ and $\nu(\G_\a) \ge \nu(\G_\b)+1 = t$. 
Since $2\nu(\G_\a) \le \sum_{i=1}^na_i$, we must have $\nu(\G_\a) = t$.  \par

To show $\nu(\G_{\a-\e_i}) = t$ for all $i \in V_\a$ we note that $\nu(\G_{\a-\e_i})\le \nu(\G_\a) = t$. 
If $i \in V_\b$,  $\G_{\a-\e_i}=\G_{\b-\e_i +\e_h+\e_j}$. Hence $\nu(\G_{\a-\e_i}) \ge \nu(\G_{\b-\e_i}) +1 = t.$
If $i \not\in V_\b$, then $i = j$ because $V_\a = V_\b \cup \{j\}$. Hence $\G_{\a-\e_i}=\G_{\b+\e_h}.$ 
Let $v$ be a vertex of $V_\b$ adjacent to $h$. 
Then $\G_{\b+\e_h} = \G_{\b-\e_v +\e_h+\e_v}$. 
Hence $\nu(\G_{\a-\e_i})  \ge \nu(\G_{\b-\e_v}) +1 = t.$
Thus, $\nu(\G_{\a-\e_i}) = t$ for all $i \in V_\a.$ \par

It remains to show that $\nu(\G_\a  - i) \ge t+1-a_i$ for all $i \in V_\a$.
If $i \in V_\b$,  we have $\nu(\G_\b  - i) \ge t - b_i$. 
For $i = h$ or $i = j$,  we have $a_i = b_i +1$. Hence $\nu(\G_\a  - i) \ge  \nu(\G_\b  - i) \ge t+1-a_i$.
For $i \neq  j, h$, we have $a_i = b_i$. Hence $\nu(\G_\a  - i) \ge  \nu(\G_\b  - i)+1 \ge t+1-a_i$. 
If $i \notin V_\b$, then $a_i = 1$ and $\G_\a  - i = \G_{\a-\e_i}.$ 
As we have seen above, $\nu(\G_{\a-\e_i}) = t$. Hence $\nu(\G_\a  - i) = t+1-a_i.$
The proof of the claim is now complete. \par

Now we continue with the proof of Lemma \ref{existence}. Since $\G_U$ is connected, we can use the claim successively to add $|U|-2s +1$ edges of $\G_U$ to $\G_\c$ to obtain a strongly weighted graph $\G_\a$ with $V_\a = U$ and  $\nu(\G_\a) = t-1$ for $t = s+(|U| - 2s+1) = |U| - s+1$.  For $t > |U| - s +1$ we only need to add more edges of $\G_U$ to get $\G_\a$.
\end{proof}

\noindent {\em Remark}.  Not all strongly $t$-saturating weighted graphs $\G_\a$ with $\sum_{i=1}^na_i = 2t+1$ and $\nu(\G_\a) = \nu(\G_{\a-\e_i})= t$ for all $i \in V_\a$ can be obtained from a strongly $s$-saturating graph on $2s-1$ vertices, $s < t$, by adding edges. The union of a triangle and a rectangle meeting only at a vertex of weight 2 is such a weighted graph. \smallskip

By Theorem \ref{embedded}, if $P_F$ is an embedded associated prime of some power of $I$, then $F$ is minimal among the covers of $\G$ containing $N[U]$ for a subset $U$ of $V$ such that every connected component of $\G_U$ contains at least one odd cycle. 
An odd cycle of length $2s-1$ is obviously a strongly $s$-saturating graph on $2s-1$ vertices. 
Therefore, using Theorem \ref{strong} and Lemma \ref{existence} we can say for which $t$ is $P_F$ an associated primes of $t$.

\begin{Theorem} \label{construction}
Let $F$ be a cover of $\G$ which is minimal among the covers containing $N[U]$ for a set $U \subseteq V$ such that every connected components of $\G_U$ contains at least one odd cycle. Assume that $\G_U$ has $m$ connected components and let $s_i$ be the largest number such that the $i$-th component of $\G_U$ in some order contains a strongly $s_i$-saturating graph on $2s_i -1$ vertices, $i = 1,...,m$. Then $P_F$ is an embedded associated prime of $I^t$ for all $t \ge |U| - \sum_{i=1}^m s_i + 1$.
\end{Theorem}

\begin{proof}
Let $\G_i$ be the $i$-th connected component of $\G_U$ and $U_i$ its vertex set, $i = 1,...,m$.  
By Lemma \ref{existence}, there exists a strongly $t_i$-saturating weighted graph $\H_i$ on $U_i$ with $\nu(\H_i) = t_i-1$ for $t_i \ge  |U_i| -s_i +1$. 
Let $\H$ be the weighted graph whose connected components are $\H_1,...,\H_m$. 
Then $\nu(\H) = \sum_{i=1}^m\nu(\H_i) = \sum_{i=1}^m(t_i-1)$. 
By Lemma \ref{strong component}, $\H$ is a strongly $t$-saturating graph for $t = \sum_{i=1}^m(t_i -1) +1$.  Now we can apply Theorem \ref{strong} to deduce that $P_F$ is an associated prime of $I^t$ 
for $t \ge \sum_{i=1}^m(|U_i|-s_i) +1 = |U|-\sum_{i=1}^m s_i + 1$. 
\end{proof}

Let $\Ass(I^t)$ denote the set of the associated primes of $I^t$.
By a result of Brodmann \cite{Br}, there exists a number $t_0$ such that $\Ass(I^t) = \Ass(I^{t+1})$ for $t \ge t_0$.
The stable set $\Ass(I^{t_0})$ is denoted by $\Ass^\infty(I)$.
Note that $\Ass(I^t) \subseteq \Ass(I^{t+1})$ for all $t \ge 1$ by Martinez-Bernal, Morey and Villarreal \cite[Theorem 2.15]{MMV}. 
 In \cite[Theorem 4.1]{CMS} Chen, Morey and Sung gave a combinatorial characterization of $\Ass^\infty(I)$ for the case $\G$ is a connected graph. However, their description of $\Ass^\infty(I)$ is recursive and too complicated to be recalled here. As an immediate consequence of Theorem \ref{embedded} and Theorem \ref{construction} we obtain the following simple characterization of $\Ass^\infty(I)$.

\begin{cor} \label{stable}
Let $F$ be a cover of $\G$. Then $P_F$ belongs to $\Ass^\infty(I)$  if and only if $F$ is a minimal cover or minimal among the covers of $\G$ containing $N[U]$ for a subset $U$ of $V$ such that every connected component of the induced graph $\G_U$ contains at least one odd cycle.  
\end{cor}

We can also give a good\ upper bound for the smallest number $t_0$ 
with the property $\Ass(I^t) = \Ass(I^{t+1})$ for $t \ge t_0$.
This number is called the {\em index of stability} of $\Ass(I^t)$ 
and denoted by $\astab(I)$ \cite{HQ}.
Note that $\astab(I) = 1$ if $\G$ is a bipartite graph 
by a result of Simis, Vasconcelos and Villarreal \cite[Theorem 5.9]{SVV}. 
\smallskip

Let $U$ be a subset of $V$ such that each connected component of $\G_U$ contains at least one odd cycle. Set $s(U) = |U|-\sum_{i=1}^ms_i +1$ if $m$ is the number of connected components of $\G_u$ and $s_i$ is the largest number such that the $i$-component (in some order) contains a strongly $s_i$-saturating graph on $2s_i-1$ vertices. Let $s(\G)$ denote the maximum of all such $s(U)$, where we set $s(\G) = 1$ if $\G$ is a bipartite graph.

\begin{cor} \label{astab}
$\astab(I)  \le s(\G)$.
\end{cor} 

\begin{proof}
Let $F$ be a cover of $\G$ such that $P_F$ is an associated prime of some of $I$ for $t \gg 0$.
Without restriction we may assume that $P_F$ is an embedded associated prime. 
By Corollary \ref{embedded}, $F$ is minimal among the covers of $\G$ containing $N[U]$ for a set $U \subseteq V$ such that every connected component of the induced graph $\G_U$ contains at least one odd cycle. 
By Theorem \ref{construction}, $P_F$ is an associated prime of $I^t$ for all $t \ge s(\G)$.
\end{proof}

Chen, Morey, and Sung \cite[Proposition 4.2]{CMS} proved that if $\G$ is connected and non-bipartite, 
then $\astab(I) \le n-s$ if $2s-1$ is the minimal length of odd cycles in $\G$ and $n > 2s-1$. 
The following example shows that the bound of Corollary \ref{astab} is much better than this bound. \smallskip

\begin{ex}
{\rm Let $\G$ be the union of $t-1$ triangles meeting each other at a common vertex, $t > 2$.
Since $\G$ has $2t-1$ vertices, $\astab(I) \le 2t-3$ by the bound of Chen, Morey, and Sung. 
On the other hand,  we have $s(\G) = t$, which implies $\astab(I) \le t$ by Corollary \ref{astab}.
To prove $s(\G) = t$ we observe that if $U$ is a subset of $V$ such that $\G_U$ contains at least one triangle,
then $\G_U$ is connected. Suppose that $\G_U$ contains $s-1$ triangles of $\G$.
Let $C$ denote the union of these triangles.
It is easy to see that $C$ is a strongly $s$-saturating graph on $2s-1$ vertices.
The remaining $|U| - (2s-1)$ vertices of $U$ must belong to different triangles of $\G$ outside of $C$. 
Hence $|U| - (2s - 1)  \le t - s$.  
Therefore, $s(U) = |U| - s + 1 \le t.$
In particular, $s(V) = (2t-1) - t + 1 = t$. So we can conclude that $s(\G) =t$.}
\end{ex}

\section{The second and third powers of edge ideals}

In this section we describe the associated primes of $I^t$ for $t = 2,3$, 
where $I$ is the edge ideal of an arbitrary simple graph $\G$.
This will be achieved by classifying the induced subgraphs $\G_U$ such that the associated primes of $I^t$ correspond to the covers of $\G$ which are minimal among the covers containing $N[U]$.
\smallskip

The associated primes of $I^2$ has been described by Herzog and Hibi \cite{HH2} and Terai and Trung \cite{TT2}.
Their results are immediate consequences of our general approach.

\begin{Theorem} \label{2 asso} \cite[Theorem 3.8]{TT2}
Let $F$ be a cover of $\G$. Then $P_F$ is an associated prime of $I^2$ if and only if $F$ is a minimal cover or minimal among the covers containing the closed neighborhood of a triangle. 
\end{Theorem}

\begin{proof}
We only need to characterize embedded associated primes of $I^2$.
Note that a 2-saturating weighted graph is a triangle. 
Then the assertion follows from Theorem \ref{associated} and Theorem \ref{strong}.
\end{proof}

\begin{cor} \label{depth 2}  \cite[Theorem 2.1]{HH2}, \cite[Theorem 2.8]{TT2}
$\depth R/I^2 > 0$ if and only if $\G$ has no dominating triangle. 
\end{cor}

\begin{proof}
Note that $\depth R/I^2  > 0$ if and only if $\mm = P_V$ is not an associated prime of $I^2$.
Since $V$ is minimal among the covers containing a subset $N$ if and only if $V = N$,
we only need to apply Theorem \ref{2 asso} to the case $F = V$ to obtain the assertion.
\end{proof}

For $t=3$ we first describe the monomials of $\widetilde{I^3} \setminus I^3$. To simplify our arguments 
we say that a weighted graph $\H$ is spanned by a weighted subgraph $\H'$ of $\H$ or $\H'$ is a {\em spanning weighted graph} of $\H$ if they share the same vertices with the same weights (their edges may be different). Moreover, we say that $\H$ is a {\em proper extension} of $\H'$ if $\H$ is not spanned by $\H'$.

\begin{Theorem} \label{3-saturating} 
Let $\a \in \NN^n$. Then $x^\a \in \widetilde{I^3} \setminus I^3$ if and only if $V_\a$ is a dominating set of $\G$ and
$\G_\a$ satisfies one of the following conditions:\par
{\rm (i) } $\G_\a$ is a triangle with weight vector $(2,2,1)$,\par
{\rm (ii)} $\G_\a$  is spanned by a union of an edge and a triangle meeting at a vertex of weight 2,\par
{\rm (iii)} $\G_\a$ is a union of two non-adjacent triangles,\par
{\rm (iv)} $\G_\a$ is spanned by the union of two triangles meeting at a vertex,\par
{\rm (v)} $\G_\a$ is spanned by a pentagon, \par
{\rm (vi)} $\G_\a$ is a complete graph $K_4$ and every vertex of $V \setminus V_\a$ is adjacent to at least two vertices of $V_\a$.
\end{Theorem}

\begin{proof}
It is easy to check that the weighted graphs $\G_\a$ listed above satisfy the conditions of Theorem \ref{monomial} for $t = 3$, which implies $x^\a \in  \widetilde{I^3}\setminus I^3$. \par

To prove the converse let $x^\a \in \widetilde{I^3}\setminus I^3$. 
Then $\G_\a$ is a 3-saturating weighted graph by Theorem \ref{monomial}, 
and $V_\a$ is a dominating set of $\G$ by Lemma \ref{dominating}. 
By the definition of 3-saturating weighted graph, $\nu(\G_\a) \le 2$.
If $\nu(\G_\a) = 1$, then $\G_\a$ is 2-saturating.
A 2-saturating weighted graph must be a triangle.
Since a triangle is not 3-saturating, we get a contradiction. 
Thus, $\nu(\G_\a) = 2$. \par

If $\G_\a$ has more than one connected components, 
then there are only two components and each component must be 2-saturating  by Lemma \ref{component}.
Hence $\G_\a$ must be a union of two disjoint triangles. 
That is Case (iii). So we may assume that $\G_\a$ is connected. \par

By Lemma \ref{odd cycle}, $\G_\a$ contains at least a cycle or a pentagon.
If $\G_\a$ contains a pentagon, then $\G_\a$ is spanned by this pentagon 
because any proper extension of a pentagon has matching number $\ge 3$. 
That is Case (v). So we may further assume that $\G_\a$ contains a triangle, say $C = \{i,j,h\}$, and no pentagon. Note that the weight of every vertex of $C$ is at most 2 by Lemma \ref{weight}(i). \par

If two vertices of $C$ has weight 2, say $a_i = a_j = 2$,
then $\G_\a$ must be the weighted graph on $C$ with weight vector $(2,2,1)$ 
because any proper extension of this weighted graph has matching number $\ge 3$. 
That is Case (i). \par

If only a vertex of $C$ has weight 2, say $a_i = 2$, 
then $\deg_\a(i) > 2$ by Lemma \ref{weight}(i). 
Since $\deg_\a(i) = \sum_{v \in N(i)}a_v$ and  $a_j = a_h = 1$, 
$N(i)$ must contain at least a vertex $v \neq j, h$. 
So $\G_\a$ contains the union of the edge $\{i,v\}$ and the triangle $C$ with only $i$ having weight 2.
Since any proper extension of this weighted graph has matching number $\ge 3$, $\G_\a$ is spanned by this weighted graph. That is Case (ii). \par

If all vertices of $C$ have weight one, then $\G_\a$ has at least a vertex $v \notin C$ 
 because $\nu(C) = 1$. Since $\G_\a$ is connected,
we may assume that $v$ is adjacent to a vertex of $C$, say $i$.
Since $a_i = 1$, $v$ is not a leaf vertex of $\G_\a$ by Lemma \ref{weight}(ii). 
So $v$ is adjacent to another vertex of $\G_\a$.
Now we distinguish two cases. \par

\noindent{\em Case 1}: $v$ is adjacent to a vertex of $\G_a$ outside of $C$, say $w$ (see Figure I).
Then $a_v = 1$ because otherwise the edges $\{w,v\}, \{v,i\}, \{j,h\}$ would form a matching of $\G_\a$, 
a contradiction to $\nu(\G_\a) = 2$. 
Thus, $w$ is not a leaf vertex of $\G_\a$ by Lemma \ref{weight}(ii). 
On the other hand, $w$ can not be adjacent to any other vertex $u \neq w$ of $\G_\a$ outside of $C$ because otherwise 
the edges $\{u,w\}, \{v,i\}, \{j,h\}$ would form a matching of $\G_\a$,  a contradiction to $\nu(\G_\a) = 2$. 
So $w$ is adjacent to at least a vertex of $C$. 
If $w$ is adjacent to $j$ or $h$, $\G_\a$ would contain a pentagon on the vertices $i,j,h,v,w$, 
a contradiction to our assumption. 
So $w$ is adjacent to $i$, and $\G_\a$ contains a union of two triangles meeting at a vertex. 
Since any proper extension of this union has matching number $\ge 3$, $\G_\a$ is spanned by this union. 
That is Case (iv).
\par

\begin{figure}[ht!]
\begin{tikzpicture}[scale=0.6]
    
\draw [thick] (-4,1) coordinate (a) -- (-2,0) coordinate (b) ;
\draw [thick] (-4,1) coordinate (a) -- (-2,2) coordinate (c) ;
\draw [thick] (-2,0) coordinate (b) -- (-2,2) coordinate (c) ; 
\draw (-2,-0.1) node[below,right] {h};
\draw (-2,2.2) node[above, right] {j};
\draw (-4.3,1) node[below, left] {i};
\draw [dashed] (-4,1) coordinate (b) -- (-6,0) coordinate (d) ;
\draw [dashed] (-2,0) coordinate (b) -- (-6,0) coordinate (d) ;
\draw [thick] (-6,0) coordinate (d) -- (-6,2) coordinate (e) ;
\draw (-6,0) node[below, left] {w};
\draw [thick] (-6,2) coordinate (e) -- (-4,1) coordinate (a) ;
\draw (-6,2) node[above,left]{v};
\draw (-4,-1) node{Figure I};
     
\fill (a) circle (3pt);
\fill (b) circle (3pt);
\fill (c) circle (3pt);
\fill (d) circle (3pt);
\fill (e) circle (3pt);

\draw [thick] (4,1) coordinate (f) -- (6,0) coordinate (g) ;
\draw [thick] (4,1) coordinate (f) -- (6,2) coordinate (h) ;
\draw [thick] (6,0) coordinate (g) -- (6,2) coordinate (h) ; 
\draw (6,-0.1) node[below,right] {h};
\draw (6,2.2) node[above, right] {j};
\draw (3.7,1) node[left] {i};
\draw [dashed] (6,0) coordinate (g) -- (2,0) coordinate (i) ;
\draw [thick] (6,2) coordinate (h) -- (2,2) coordinate (j) ;
\draw (2,0) node[below, left] {w};
\draw [thick] (2,2) coordinate (j) -- (4,1) coordinate (f) ;
\draw [dashed] (2,0) coordinate (i) -- (4,1) coordinate (f) ;
\draw (2,2) node[above,left]{v};
\draw (4,-1) node{Figure II};
     
\fill (f) circle (3pt);
\fill (g) circle (3pt);
\fill (h) circle (3pt);
\fill (i) circle (3pt);
\fill (j) circle (3pt);

\end{tikzpicture}
\end{figure}

\noindent{\em Case 2}:  $v$ is not adjacent to any vertex of $\G_\a$ outside of $C$.
Then $v$ is adjacent to another vertex of $C$, say $j$ (see Figure II).
Let $w$ be an arbitrary vertex of $\G_\a$ outside of $C$ which is adjacent to a vertex of $C$.
By Case 1, we may assume that $w$ is not adjacent to any vertex of $\G_\a$ outside of $C$. 
Since $\G_\a$ is connected, this implies that $\G_\a - C$ consists of only isolated vertices.
As a consequence, $\nu(\G_\a-N_\a(h)) = 0$ because $C \subseteq N_\a(h)$.
By Theorem \ref{monomial}(ii), this implies $\deg_\a(h) \ge 3$.
Hence $N(h)$ contains at least a vertex of $\G_\a$ outside of $C$. Let $w$ be such a vertex.
Like $v$, the vertex $w$ is adjacent to at least a vertex of $C$ other than $h$, say $i$.
If $w \neq v$, we would get a pentagon on the vertices $i,v,j,h,w$, a contradiction to our assumption. 
So we must have $w = v$. Since $v, w$ are arbitrarily chosen, $v$ must be the unique vertex of $\G_\a$ outside of $C$. Therefore, $\G_\a$ is a complete graph $K_4$.
For every vertex $u \in V \setminus V_\a$ we have
$\deg_\a(u) + \nu(\G_\a - N_\a(u)) \ge 3$ 
by Theorem \ref{monomial}(ii).
Since $V_\a$ is a dominating set of $\G$, $N_\a(u) \neq \emptyset$. 
Consequently, $\G_\a - N_\a(u)$ is contained in a triangle. Hence
 $\nu(\G_\a - N_\a(u)) \le 1$. So we get $\deg_\a(u) \ge 2$. 
This means $u$ is adjacent to at least two vertices of $\G_\a$. 
That is Case (vi). \par

The proof of Theorem \ref{3-saturating} is now complete.
\end{proof}

Using Theorem \ref{3-saturating} we can characterize the associated primes of $I^3$ as follows.

\begin{Theorem} \label{asso 3} 
Let $F$ be a cover of $\G$. Then $P_F$ is an associated prime of $I^3$ if and only if $F$ is a minimal cover or $F$ is minimal among the covers containing the closed neighborhood of the vertex set of a subgraph of the following forms:
a triangle, a union of an edge and a triangle meeting at a vertex, a union of two non-adjacent triangles,  a union of two triangles meeting at a vertex, a pentagon. 
\end{Theorem}

\begin{proof}
We may assume that $F$ is not a minimal cover of $\G$. Then $\core(F) \neq \emptyset$. 
Let $S := k[x_i|\ i \in \core(F)]$ and $J := I(\G_{\core(F)})$.  
By Theorem \ref{monomial} and Theorem \ref{associated},
$P_F$ is an associated prime of $I^3$ if and only if $F$ is a minimal  among the covers containing $N[V_\a]$ for some $\a \in \NN^n$ such that $x^\a \in \widetilde{J^3} \setminus J^3$. 
By Theorem \ref{3-saturating}, this is the case if and only if the base graph 
of $\G_\a$ is of the above forms. Note that if $\G_\a$ is a complete graph $K_4$ and every vertex of $\core(F) \setminus V_\a$ is adjacent to two vertices of $V_\a$, then every vertex of $\core(F) \setminus V_\a$ is adjacent to every triangle of this $K_4$. Hence this case can be passed to the case of a triangle.
\end{proof}

As in the proof of Corollary \ref{depth 2} we obtain from Theorem \ref{asso 3} the following criterion.

\begin{cor} \label{depth 3} 
$\depth R/I^3 > 0$ if and only if $\G$ has no dominating subgraph of the following forms: 
a triangle, a union of an edge and a triangle meeting at a vertex, a union of two non-adjacent triangles, a union of two triangles meeting at a vertex, a pentagon. 
\end{cor}

\end{document}